  \documentclass[review,10pt,oneside,a4paper]{elsarticle}

\usepackage{latexsym}
\usepackage{amscd}
\usepackage{amssymb}
\usepackage{amstext}
\usepackage{amsmath}
\usepackage{amsxtra}
   \usepackage[hang,centerlast,small]{subfigure}
   \usepackage{amsthm}
   \newcommand*{\ud}{\hspace{0.1cm}\mathrm{d}}
      
   \newcommand{\gf}{\mathcal{G}}

   \newcommand{\bR}{\mathbb{R}}

   \newtheorem{theorem}{Theorem}[section]
   \newtheorem{lemma}[theorem]{Lemma}
   \newtheorem{definition}[theorem]{Definition}
   \newtheorem{corollary}[theorem]{Corollary}
   \newtheorem{proposition}[theorem]{Proposition}
   \newtheorem{remark}[theorem]{Remark}
   \numberwithin{equation}{section}
   \def\!{\mskip-0.6\thinmuskip}
   \newcommand{\snorm}[1]{
   \bgroup\left\vert#1\right\vert\egroup}
   \newcommand{\dnorm}[1]{
   \bgroup\left\vert\!\left\vert#1\right\vert\!\right\vert\egroup}
   \newcommand{\tnorm}[1]{
   \bgroup\left\vert\!\left\vert\!\left\vert#1\right\vert\!\right\vert\!\right\vert\egroup}
   \newcommand{\iprod}[2]{\bgroup\left<#1 , #2\right>\egroup}
   \usepackage[displaymath,mathlines,left]{lineno}
   \usepackage{xcolor}

   \newcommand{\markred}[1]{\bgroup\color{red}#1\egroup}

   \newcommand{\markblue}[1]{\bgroup\color{blue}#1\egroup}
\makeatletter
   \def\ps@pprintTitle{%
   \let\@oddhead\@empty
   \let\@evenhead\@empty
   \def\@oddfoot{}%
   \let\@evenfoot\@oddfoot}
\makeatother

\begin{document}

\begin{frontmatter}
\title{ 
\textbf{
Nonlinear sequential fractional boundary value problems involving generalized $\psi$-Caputo fractional derivatives
}
}	

\author[affil1]{Nguyen Minh Dien \corref{corresponding-author}}
\cortext[corresponding-author]{Corresponding author}
\ead{diennm@tdmu.edu.vn; nmdien81@gmail.com}

\address[affil1]{Faculty of Education, Thu Dau Mot University, Binh Duong Province, Vietnam}

\begin{abstract}
This paper is devoted to study the nonlinear sequential fractional boundary value problems involving generalized $\psi$-Caputo fractional derivatives with nonlocal boundary conditions. We investigate the Green function and some of its properties, from which we obtain a new Lyapunov-type inequality for our problem. A lower bound for the possible eigenvalues of our problem is derived. Furthermore, we apply some properties of the Green function to obtain some existence results for our problem. It is worth mentioning that our results still work with some source functions including singularities.
\end{abstract} 

\end{frontmatter}

\section{Introduction}

\subsection{Statement of the problem}
Let $a<b$, $0<\alpha, \beta \le 1$ with $\alpha+\beta>1$, and let $\psi$ be an appropriate function.
We consider the sequential fractional boundary value problem
involving generalized $\psi$-Caputo fractional derivatives
\begin{equation}
\label{main}
\left(
   {^C}D_{a+}^{\alpha, \psi} 
   \,
   {^C}D_{a+}^{\beta, \psi}
   x
\right)
(t)
+
f(t, x(t))
=0, \ \  a<t<b
\end{equation}
subject to the generalized nonlocal conditions 
\begin{equation}
\label{BV}
x(a)=0, \ x(b)=\gf(x)
,
\end{equation}
where ${^C}D_{a+}^{\alpha, \psi}$, ${^C}D_{a+}^{\beta, \psi}$ are the left $\psi$-Caputo fractional derivatives.
~
~
Our study focuses on discussing the following aspects: 
the Green function and some of its properties, 
the Lyapunov-type inequality, 
and some existence results for our problem.

\subsection{Relevant works and motivations}

Fractional calculus is a generalization of ordinary differentiation and integration, in which derivatives and integrals of arbitrary real or complex order are defined.
This field has became a rapidly growing area and has found applications in diverse research fields such as 
in physics \cite{Her14}, fractional dynamics \cite{Tar10}, quantum mechanics \cite{Iom19}, bioengineering \cite{Mag06}.
Particularly, one can gain a comprehensive view of applications of the fractional derivatives in physics from the excellent monograph \cite{Uch13} and the references therein.
~
~
Definition of fractional derivatives are very rich such as Caputo, Caputo-Katugampola, Riemann-Liouville, Hadamard, etc.
The concept of fractional derivatives of a function with respect another function was introduced in \cite{KST06}.
Based on this ideas, Almeida \cite{Alm17} presented the $\psi$-Caputo fractional derivatives which is generalized from the mentioned fractional derivatives.

Sequential fractional differential equations were firstly introduced in the literature in the book by Miller and Ross \cite{MR93}, in which the authors defined compositions of Riemann-Liouville fractional derivatives
$
\prod_{i=1}^n D^\alpha = \underbrace{D^\alpha D^\alpha..D^\alpha}_{n \ times}
$,
where $n \in \mathbb{N}$ and $\alpha$ is a positive real number. 
They also investigated some linear differential equations involving these compositions.

In the past few years, there are numerous works on the sequential fractional boundary value problems
such as \cite{ABR21, AAB21, Dien-Fractals-21, Dien-RMJ-21, Dien-JIEA-20, GGOR20, FN18, Fer19, Fer20, KT21, UBA20, ZL19} and the references therein.
~
~
In fact, 
Ferreira \cite{Fer19} and Zhang al et \cite{ZL19} derived some Lyapunov-type inequalities for linear sequential fractional boundary value problems involving Caputo fractional derivatives and Hilfer fractional derivatives, respectively.
~
~
Fazli and Nieto \cite{FN18} were obtained some existence and uniqueness results for the nonlinear initial value problems with Riemann-Liouville fractional derivatives $D^{2 \alpha}=D^{\alpha}D^{\alpha}$.
Kassymov and Torebek \cite{KT21} obtained a Lyapunov-type and a Hartman-Wintner-type inequalities for a
nonlinear fractional hybrid equation with left Riemann-Liouville and right Caputo fractional derivatives.
Ferreira \cite{Fer20} also obtained a Lyapunov-type inequality and a existence result for a nonlinear fractional derivatives with Riemann-Liouville and the Caputo fractional derivatives.
~
~
Very recently, in \cite{Dien-Fractals-21, Dien-RMJ-21, Dien-JIEA-20}, we investigated the existence and continuity results for the fractional Langevin problems involving $\psi$-Caputo fractional derivatives with weakly singular sources.

To the best of our knowledge, Lyapunov-type inequalities for nonlinear sequential fractional differential equations involving $\psi$-Caputo fractional derivatives associated with generalized nonlocal conditions are still not considered. Furthermore, the existence results for the problem including some singularities in the source function are under consideration. 
~
~
Motivated by the above discussions, in this paper, we consider the sequential fractional boundary value problem (\ref{main}) and (\ref{BV}). 
In our work, the main new features are that:

\begin{itemize}

\item[•]
We establish the Green function and investigate some of its properties.
It is worth noting that the Green function in the present paper is different from the Green function in recent papers.

\item[•]
We derive a new Lyapunov-type inequality for our problem. 

\item[•]
We obtain some existence results for our problem.
Unlike previous papers, our obtained results still work with some source functions including singularities.
\end{itemize}

\subsection{Outline}

The present paper is organized as follows.
In section 2, we introduce the reader to some basic concepts of fractional integral and $\psi$-Caputo fractional derivative. Some well-known results are also recalled before proceeding to the main results.
In section 3, we establish the Green function and some of its properties.
In section 4, we present a new Lyapunov-type inequality and some existence results for our problem.

\section{Preliminaries}
We begin by setting some notations.
For convenience, we denote the class of increasing and differentiable functions on $[a, b]$ by $C^1_+[a, b]$, i.e.,
\[
C^1_+[a, b]
=
\{\psi \in C^1[a, b]: \psi'(t)>0 \ \text{for all} \ t \in [a, b]\}
.
\]
For $z$ belong to $C([a, b], \bR)$, we define by $\dnorm{z}=\sup_{a \le t \le b}|z(t)|$.
We also denote the set of non-negative real numbers by $\bR_+$.

To present the concept of $\psi$-Caputo fractional derivative, we start with the definition of fractional integral
of a function with respect to another function.

\begin{definition}[see \cite{Alm17, KST06}]
For $\alpha>0$, $\psi \in C^1_+[a, b]$, and $x \in L^1[a, b]$, 
the left fractional integral of a function $x$ depending another function $\psi$ is
given by 
\[
I_{a+}^{\alpha, \psi} x(t)
=
\frac{1}{\Gamma(\alpha)}
\int_a^t 
   \psi'(\tau)(\psi(t)-\psi(\tau))^{\alpha-1} 
   x(\tau)
\ud \tau
,
\]
where $\Gamma(\cdot)$ denotes the Gamma function.
\end{definition}

\begin{definition}[see \cite{Alm17}]
For $n-1 < \alpha \le n$, and $x, \psi \in C^n[a, b]$ with $\psi'(t)>0$ for all $t \in [a, b]$,
the left-side $\psi$-Caputo fractional derivative of $x$ of order $\alpha$ is defined by
\begin{equation*}
{^C} D^{\alpha, \psi}_{a+} x(t)
=
I_{a+}^{n-\alpha, \psi} 
\left(\frac{1}{\psi'(t)}\frac{\ud}{\ud t}\right)^n
x(t)
.
\end{equation*}
\end{definition}

For complete surveys of basic properties of the fractional operators 
$ 
   I_{a+}^{\alpha, \psi} 
$ 
and 
$
   ^C D^{\psi, \alpha}_{a+} 
$, 
we refer to \cite{Alm17, KST06}. In this paper, we will use only the following properties.

\begin{lemma}
[see \cite{Alm17, KST06}]
\label{lem-CP}
We have

(i).
For $x, y \in C^n[a, b]$ then
$
{^C} D^{\alpha, \psi}_{a+} x(t)
=  \,
{^C} D^{\alpha, \psi}_{a+} y(t)
$
if and only if
$
x(t)
=
y(t)+\sum_{k=0}^{n-1} c_k  (\psi(t)-\psi(a))^k
$.

(ii).
For $x \in C^1[a, b]$, $\alpha>0$ then
$
{^C} D^{\alpha, \psi}_{a+} I^{\alpha, \psi}_{a+} x(t)
=
x(t)
$.

(iii). For $\alpha, \beta>0$, we have
$
I^{\alpha, \psi}_{a+} I^{\beta, \psi}_{a+} x(t)
=
I^{\alpha+\beta, \psi}_{a+} x(t)
$.
\end{lemma}

At the end of this section, we present two well-known results that we will use to prove the main results of our paper.

\begin{lemma}[Jensen's inequality]
\label{Jen}
Let $\mu$ be a positive measure and let $\Omega$ be a measurable set with $\mu(\Omega)=1$.
If $x$ is a real function in $L^1(\mu)$, if $a<x(t)<b$ for all $t \in \Omega$, and if $f$ is a convex on $(a, b)$, then
\begin{equation}
\label{covex}
f\left(\int_\Omega x \ud \mu\right)
\le 
\int_\Omega (f \circ x )\ud \mu
.
\end{equation}
If $f$ is concave on $(a, b)$, then the inequality \eqref{covex} holds with $\le$ reversed.
\end{lemma}

\begin{lemma}[The nonlinear Leray-Schauder alternatives fixed point theorem \cite{Gra03}]
\label{Le-Sch}
Let $\mathbb{B}$ be a Banach space, and let $W$ be a closed convex subset of $\mathbb{B}$.
Let $V$ be a relatively open subset of $W$ and $0 \in V$.
Suppose that $Q: \overline{V} \to W$ is a continuous compact mapping. 
Then we have either

(i). $Q$ has a fixed point in $\overline{V}$

\ \ \ or

(ii). There exist $\lambda \in (0, 1)$ and $u \in \partial V$ such that
$u=\lambda Q u$.
\end{lemma}

\section{Green function and some of its properties}
In this section, we establish the Green function for our problem and some of its properties.
We start with constructing the Green function for our problem.

\begin{lemma}
\label{lem-GR}
Let $0<\alpha, \beta \le 1$ with $\alpha+\beta>1$,
and let $\psi \in C^1_+[a, b]$.
If $x$ is a solution of the problem (\ref{main}) and (\ref{BV}) such that
${^C}D_{a+}^{\beta, \psi} x$, $f(\cdot, x(\cdot)) \in C^1[a, b]$, then $x$ 
is a solution of the following non-local integral equation
\begin{equation}
\label{Int-eq}
x(t)
=
\left(\frac{\psi(t)-\psi(a)}{\psi(b)-\psi(a)}\right)^{\beta}
\gf(x)
+
\int_a^b G(\tau, t) \psi'(\tau) f(\tau, x(\tau)) \ud \tau
\end{equation}
where
\begin{equation}
\label{GR-df}
G(\tau, t)
=
\begin{cases}
   g_1(\tau, t),
   \ \ & a \le \tau \le t \le b,
   \cr
   g_2(\tau, t),
   \ \ & a \le t \le \tau \le b
   ,
\end{cases}
\end{equation}
where
\begin{align*}
g_1(\tau, t)
& =
C (\psi(t)-\psi(a))^\beta (\psi(b)-\psi(\tau))^{\alpha+\beta-1}
-
D (\psi(t)-\psi(\tau))^{\alpha+\beta-1}
,
\\
g_2(\tau, t)
& =
C (\psi(t)-\psi(a))^\beta (\psi(b)-\psi(\tau))^{\alpha+\beta-1}
,
\end{align*}
and
$
C
=
1
/
\left[
   \Gamma(\alpha+\beta) 
   (\psi(b)-\psi(a))^\beta
\right]
$,
$ 
D
=
1/\Gamma(\alpha+\beta)
$.
\end{lemma}

\begin{proof}
From the assumption $f(\cdot, x(\cdot)) \in C^1[a, b]$, we obtain from Lemma \ref{lem-CP} that
\[
   {^C}D_{a+}^{\alpha, \psi} I_{a+}^{\alpha, \psi} f(t, x(t))
   =
   f(t, x(t))
   .
\]
This leads to Eq. (\ref{main}) equivalent to the following equation
\[
\left(
   {^C}D_{a+}^{\alpha, \psi} 
   \,
   {^C}D_{a+}^{\beta, \psi}
   x
\right)
(t)
=
- {^C}D_{a+}^{\alpha, \psi} I_{a+}^{\alpha, \psi} f(t, x(t))
.
\]
By virtue of Lemma \ref{lem-CP}, one has
\begin{align}
\label{c12}
x(t)
=
c_2
+
\left(I_{a+}^{\beta, \psi} c_1\right) (t)
 - 
I_{a+}^{\alpha+\beta, \psi} f(t, x(t))
 =
c_2 + \frac{c_1}{\Gamma(\beta+1)} (\psi(t)-\psi(a))^\beta
+
I_{a+}^{\alpha+\beta, \psi} f(t, x(t))
.
\end{align}
Here we have used the fact that 
$
\left(I_{a+}^{\beta, \psi} c_1\right) (t)
=
(c_1/\Gamma(\beta+1)) (\psi(t)-\psi(a))^\beta
$.

We now find the coefficients $c_1, c_2$.
Since $u(a)=0$, we find out $c_2=0$.
Continuously, we have
\begin{align*}
x(b)
 =
\frac{c_1}{\Gamma(\beta+1)}
(\psi(b)-\psi(a))^\beta
- 
\frac{1}{\Gamma(\alpha+\beta)}
\int_a^b \psi'(\tau) (\psi(b)-\psi(\tau))^{\alpha+\beta-1} f(\tau, x(\tau)) \ud \tau
.
\end{align*}
Using the nonlocal condition $x(b)=\gf(x)$, we get
\begin{align*}
c_1
=
\frac{
   \Gamma(\beta+1)
}
{
   (\psi(b)-\psi(a))^{\beta}
} 
\left(
\gf(x)
+
\frac{1}{\Gamma(\alpha+\beta)}
\int_a^b
   \psi'(\tau) (\psi(b)-\psi(\tau))^{\alpha+\beta-1} 
   f(\tau, x(\tau))
\ud \tau
\right)
.
\end{align*}
Substituting the obtained coefficients $c_1$ and $c_2$ into (\ref{c12}), we obtain
\begin{align*}
x(t)
&=
\left(\frac{\psi(t)-\psi(a)}{\psi(b)-\psi(a)}\right)^{\beta}
\gf(x)
+
\frac{1}{\Gamma(\alpha+\beta)} 
\left(\frac{\psi(t)-\psi(a)}{\psi(b)-\psi(a)}\right)^{\beta}
\int_a^b
   \psi'(\tau) (\psi(b)-\psi(\tau))^{\alpha+\beta-1} 
   f(\tau, x(\tau))
\ud \tau
\\
& -
\frac{1}{\Gamma(\alpha+\beta)}
\int_a^t \psi'(\tau) (\psi(t)-\psi(\tau))^{\alpha+\beta-1} f(\tau, x(\tau)) \ud \tau
\\
& =
\left(\frac{\psi(t)-\psi(a)}{\psi(b)-\psi(a)}\right)^{\beta}
\gf(x)
+
\int_a^b G(\tau, t) \psi'(\tau) f(\tau, x(\tau)) \ud \tau
.
\end{align*}
The proof of Lemma is completed.
\end{proof}

\begin{definition}
The function $G(\cdot, \cdot)$ given by (\ref{GR-df}) is called the Green function of the problem
(\ref{main}) and (\ref{BV}).
\end{definition}

\begin{proposition}
\label{pro-max}
Let $0<\alpha, \beta \le 1$ with $\alpha+\beta>1$, and let $\psi \in C^1_+[a, b]$ and the Green function $G$ be defined as in Lemma \ref{lem-GR}.
Then
\begin{align*}
\max_{a \le \tau, t \le b}|G(\tau, t)|
=
\frac{1}{\Gamma(\alpha+\beta)}
(\psi(b)-\psi(a))^{\alpha+\beta-1}
\max\left\{
\beta^\beta (\alpha+\beta-1)^{\alpha+\beta-1}
,
\frac{\beta (\alpha+\beta-1)^{\frac{\alpha+\beta-1}{\beta}}}{(\alpha+2\beta-1)^{\frac{\alpha+2\beta-1}{\beta}}} 
\right\}
.
\end{align*}
\end{proposition}

\begin{proof}
We consider the functions $g_1$ and $g_2$ 
which are defined in Lemma \ref{lem-GR}.
We divide the proof into two steps.
\\
\\
{\it Step 1.}
We show that
\begin{equation*}
\max_{a \le t \le \tau \le b}
g_2(\tau, t)
=
g(\tau_0)
=
\frac{1}{\Gamma(\alpha+\beta)} 
\beta^\beta (\alpha+\beta-1)^{\alpha+\beta-1}
(\psi(b)-\psi(a))^{\alpha+\beta-1}
.
\end{equation*}

Since $\psi$ is an increasing function, hence, we have $0 \le \psi(t)-\psi(a) \le \psi(\tau)-\psi(a)$
for all $a \le t \le \tau \le b$. This leads to
\begin{equation}
\label{g-df}
0 \le 
g_2(\tau, t)
\le 
C (\psi(\tau)-\psi(a))^\beta (\psi(b)-\psi(\tau))^{\alpha+\beta-1}
:=g(\tau)
\end{equation}
for any $a \le t \le \tau \le b$.
~
~
By direct computations, we have 
\begin{align*}
g'(\tau)
=
C(\psi(\tau)-\psi(a))^{\beta-1} (\psi(b)-\psi(\tau))^{\alpha+\beta-2}
\left(\beta (\psi(b)-\psi(\tau))-(\alpha+\beta-1)(\psi(\tau)-\psi(a))\right)
,
\end{align*}
where $C$ defined in Lemma \ref{lem-GR}.
Since $\beta-1 \le 0$ and $\alpha+\beta-2 \le 0$, 
we can find that $g'(\tau)=0$ at
$
\tau_0
=
\psi^{-1}\left(\frac{\beta \psi(b)+(\alpha+\beta-1)\psi(a)}{\alpha+2\beta-1}\right)
$.
Observe that $\psi(a)<\frac{\beta \psi(b)+(\alpha+\beta-1)\psi(a)}{\alpha+2\beta-1}<\psi(b)$,
this implies that $\tau_0 \in (a, b)$ and
$
\max_{a \le \tau \le b} g(\tau)
=
g(\tau_0)
=
C \beta^\beta (\alpha+\beta-1)^{\alpha+\beta-1}
(\psi(b)-\psi(a))^{\alpha+2\beta-1}
$.
This leads to the desired result of { \it Step 1}.
\\
\\
{\it Step 2.}
We prove that
\begin{align*}
\max_{a \le \tau \le t \le b} |g_1(\tau, t)|
=
\frac{1}{\Gamma(\alpha+\beta)}
(\psi(b)-\psi(a))^{\alpha+\beta-1}
\max\left\{
\beta^\beta (\alpha+\beta-1)^{\alpha+\beta-1}
,
\frac{\beta (\alpha+\beta-1)^{\frac{\alpha+\beta-1}{\beta}}}{(\alpha+2\beta-1)^{\frac{\alpha+2\beta-1}{\beta}}} 
\right\}
.
\end{align*}

We firstly verify that $g_1(\tau, t)$ is an increasing function with respect to $\tau$ on $[a, t]$.
~
~
Indeed, using the fact that  $C/D=(\psi(b)-\psi(a))^\beta$
and the following estimate
\begin{align*} 
\left(\frac{\psi(t)-\psi(a)}{\psi(b)-\psi(a)}\right)^\beta 
(\psi(b)-\psi(\tau))^{\alpha+\beta-2}
-
(\psi(t)-\psi(\tau))^{\alpha+\beta-2}
 \le
(\psi(b)-\psi(\tau))^{\alpha+\beta-2}
-
(\psi(t)-\psi(\tau))^{\alpha+\beta-2}  
\le 0  
,
\end{align*}
we get
\begin{align*}
\frac{\partial g_1}{\partial \tau}(\tau, t)
& = 
- (\alpha+\beta-1) \psi'(\tau)
\left[
   C (\psi(t)-\psi(a))^\beta (\psi(b)-\psi(\tau))^{\alpha+\beta-2}
   -
   D (\psi(t)-\psi(\tau))^{\alpha+\beta-2}
\right]
\\
& = 
- D(\alpha+\beta-1) \psi'(\tau)
\left[
   \left(\frac{\psi(t)-\psi(a)}{\psi(b)-\psi(a)}\right)^\beta 
   (\psi(b)-\psi(\tau))^{\alpha+\beta-2}
   -
   (\psi(t)-\psi(\tau))^{\alpha+\beta-2}
\right]
\\
& \ge 0 \ \ \text{for all} \ \ \tau \in [a, t]
.
\end{align*}
This implies that
\begin{equation*}
\max_{a \le \tau \le t}|g_1(\tau, t)|
=
\max\{|g_1(a, t)|, |g_1(t, t)|\}
.
\end{equation*}
Thus,
\begin{equation}
\label{es-g1}
\max_{a \le \tau \le t \le b} |g_1(\tau, t)|
=
\max\{\max_{a \le t \le b}|g_1(a, t)|, \max_{a \le t \le b}|g_1(t, t)|\}
.
\end{equation}
Note that $g_1(t, t)=g(t)$ with $g$ defined in (\ref{g-df}), this gives $g_1(t, t) \ge 0$ and
\begin{equation}
\label{es1-g1}
\max_{a \le t \le b}|g_1(t, t)|
=
\frac{1}{\Gamma(\alpha+\beta)} 
\beta^\beta (\alpha+\beta-1)^{\alpha+\beta-1}
(\psi(b)-\psi(a))^{\alpha+\beta-1}
.
\end{equation}
We now consider the function $h(t):=g_1(a, t)$.
By directly computes, one has
\[
h'(t)
=
\frac{(\psi(t)-\psi(a))^{\alpha+\beta-2}}{\Gamma(\alpha+\beta)} 
\left(
   (\alpha+2\beta-1) \left(\frac{\psi(t)-\psi(a)}{\psi(b)-\psi(a)}\right)^\beta
   -
   (\alpha+\beta-1)
\right)
.
\]
We deduce $h'(t)=0$ at $t_0$ such that 
$
\psi(t_0)
=
\psi(a)
+
\left(\frac{\alpha+\beta-1}{\alpha+2\beta-1}\right)^{1/\beta}
(\psi(b)-\psi(a))
:=\psi_0
$,
or 
$t_0=\psi^{-1}(\psi_0)$.
Since $0<\left(\frac{\alpha+\beta-1}{\alpha+2\beta-1}\right)^{1/\beta}<1$, 
then we have 
$
\psi(a)
<
\psi_0
=
\psi(a)
+
\left(\frac{\alpha+\beta-1}{\alpha+2\beta-1}\right)^{1/\beta}
(\psi(b)-\psi(a))
<\psi(b)
$.
This gives $a<t_0<b$.
On the other hand, we have $h(a)=0$ and 
\begin{align*}
h(t)
=
\frac{1}{\Gamma(\alpha+\beta)} 
(\psi(t)-\psi(a))^{\alpha+\beta-1}
\left(
   \left(\frac{\psi(t)-\psi(a)}{\psi(b)-\psi(a)}\right)^\beta - 1
\right)
\le 
0
\end{align*}
for all $t \in [a, b]$.
From the above discussions, we conclude that 
\begin{align}
\label{es2-g1}
\max_{a \le t \le b} |h(t)|
=
|h(t_0)|
=
\frac{\beta}{(\alpha+2\beta-1)\Gamma(\alpha+\beta)} 
\left(\frac{\alpha+\beta-1}{\alpha+2\beta-1}\right)^{\frac{\alpha+\beta-1}{\beta}}
(\psi(b)-\psi(a))^{\alpha+\beta-1}
.
\end{align}
Combining (\ref{es1-g1}) and (\ref{es2-g1}) together with (\ref{es-g1}),
we obtain the result of {\it Step 2}.

Finally, we can combine {\it Step 1} and {\it Step 2} to obtain the desired result of Proposition. 
\end{proof}

\begin{proposition}
\label{pro-conti}
Let $0<\alpha, \beta \le 1$ with $\alpha+\beta>1$, and let $\psi \in C^1_+[a, b]$ and the Green function $G$ be defined as in Lemma \ref{lem-GR}.
Then, for any $a \le t_1 \le t_2 \le b$, we have
\[
|G(\tau, t_2)-G(\tau, t_1)|
\le 
\frac{2}{\Gamma(\alpha+\beta)} 
(\psi(t_2)-\psi(t_1))^{\alpha+\beta-1}
.
\]
\end{proposition}

\begin{proof}
In the process of the proof of Proposition, we use the fact that
\begin{equation}
\label{inq}
|v^p-w^p| \le |v-w|^p
\end{equation}
for any $v, w \ge 0$ and $0 \le p \le 1$.
\\
We will consider three cases.
\\
{\it The first case: $a \le t_1 \le t_2 \le \tau \le b$.}
We firstly have
$
(\psi(t_2)-\psi(t_1))^\beta
=
(\psi(t_2)-\psi(t_1))^{\alpha+\beta-1}
(\psi(t_2)-\psi(t_1))^{1-\alpha}
\le 
(\psi(t_2)-\psi(t_1))^{\alpha+\beta-1}
(\psi(b)-\psi(a))^{1-\alpha}
$.
Hence, we obtain from (\ref{inq}) that
\begin{align*}
|G(\tau, t_2)-G(\tau, t_1)|
& \le 
C (\psi(b)-\psi(\tau))^{\alpha+\beta-1}
\left|(\psi(t_2)-\psi(a))^\beta-(\psi(t_1)-\psi(a))^\beta\right|
\\
& \le 
C (\psi(b)-\psi(a))^{\alpha+\beta-1}
(\psi(t_2)-\psi(t_1))^\beta
\\
& \le 
C (\psi(b)-\psi(a))^\beta
(\psi(t_2)-\psi(t_1))^{\alpha+\beta-1}
\\
& =
\frac{1}{\Gamma(\alpha+\beta)}
(\psi(t_2)-\psi(t_1))^{\alpha+\beta-1}
\end{align*}
due to $C(\psi(b)-\psi(a))^\beta=1/\Gamma(\alpha+\beta)$.
\\
{\it The second case: $a \le \tau \le t_1 \le t_2 \le b$.}
Similar to the first case, we have
\begin{align*}
|G(\tau, t_2)-G(\tau, t_1)|
& \le 
C (\psi(b)-\psi(\tau))^{\alpha+\beta-1}
\left|(\psi(t_2)-\psi(a))^\beta-(\psi(t_1)-\psi(a))^\beta\right|
\\
& +
D \left|(\psi(t_2)-\psi(\tau))^{\alpha+\beta-1}-(\psi(t_1)-\psi(\tau))^{\alpha+\beta-1}\right|
\\
& \le 
\frac{2}{\Gamma(\alpha+\beta)}
(\psi(t_2)-\psi(t_1))^{\alpha+\beta-1}
\end{align*}
due to $D=1/\Gamma(\alpha+\beta)$.
\\
{\it The third case: $a \le t_1 \le \tau \le t_2 \le b$.}
Using the same methods to that have used to prove the first case, we have
\begin{align*}
|G(\tau, t_2)-G(\tau, t_1)|
& \le 
C (\psi(b)-\psi(\tau))^{\alpha+\beta-1}
\left|(\psi(t_2)-\psi(a))^\beta-(\psi(t_1)-\psi(a))^\beta\right|
+
D(\psi(t_2)-\psi(\tau))^{\alpha+\beta-1}
\\
& \le 
\frac{2}{\Gamma(\alpha+\beta)}
(\psi(t_2)-\psi(t_1))^{\alpha+\beta-1}
\end{align*}
due to $\psi(t_2)-\psi(\tau) \le \psi(t_2)-\psi(t_1)$.
Combining three cases above, we obtain the desired result of Proposition.
\end{proof}

\section{Lyapunov-type inequality and existence results}
In this section, we investigate a new Lyapunov-type inequality and obtain a lower bound for the possible eigenvalues of our problem. We also present some existence results for our problem.

To obtain Lyapunov-type inquality, the following assumptions will be posed.
\begin{itemize}
\item[•]
{\bf  Assumption $(\mathcal{A} 1)$:} 
There exist a positive function $q: (a, b) \to \bR_+$ 
and a positive, non-decreasing and concave function  $\phi: \bR \to \bR$ such that
\[
| f(t, x) |
\le 
q(t) |\phi(x)|
\]
for any $t \in (a, b)$ and $x \in \bR$.

\item[•]
{\bf  Assumption $(\mathcal{A} 2)$:} 
There exists $\kappa \in (0, 1)$ such that
\[
|\gf(x)|
\le 
\kappa |x|
\]
for any $x \in \bR$.
\end{itemize}

In this section, we also use the concept of mild solutions as follows.

\begin{definition}
The function $x \in C[a, b]$ satisfying the Eq. (\ref{Int-eq}) is called mild solution of the problem (\ref{main}) and (\ref{BV}).
\end{definition}

We now present a Lyapunov-type inequality for our problem.

\begin{theorem}
\label{thm-Lya}
Let $0<\alpha, \beta \le 1$ with $\alpha+\beta>1$, and let $\psi \in C^1_+[a, b]$.
Suppose that Assumptions $(\mathcal{A} 1)$ and $(\mathcal{A} 2)$ are satisfied. 
If $\psi'(\cdot) q(\cdot) \in L^1(a, b)$ and the problem (\ref{main})-(\ref{BV}) has a nontrivial mild solution,
then
\[
\int_a^b \psi'(\tau) q(\tau) \ud \tau
\ge 
\frac{1-\kappa}{G_{\max}}
\frac{\|x\|}{\phi(\|x\|)}
,
\]
where
\begin{align*}
G_{\max}
=
\frac{1}{\Gamma(\alpha+\beta)}
(\psi(b)-\psi(a))^{\alpha+\beta-1}
\max\left\{
\beta^\beta (\alpha+\beta-1)^{\alpha+\beta-1}
,
\frac{\beta (\alpha+\beta-1)^{\frac{\alpha+\beta-1}{\beta}}}{(\alpha+2\beta-1)^{\frac{\alpha+2\beta-1}{\beta}}} 
\right\}
.
\end{align*} 
\end{theorem}

\begin{remark}
Lyapunov-type inequality obtained in Theorem \ref{thm-Lya}
seems to be new, which is still not proposed in previous papers.
\end{remark}

\begin{proof}
Using the fact that $\frac{\psi(t)-\psi(a)}{\psi(b)-\psi(a)}\le 1$ for all $t \in [a, b]$,
Lemma \ref{Jen} and Lemma \ref{pro-max}, we obtain from (\ref{Int-eq}) that
\begin{align*}
|x(t)|
& \le 
|\gf(x)|
+
G_{\max}
\int_a^b  \psi'(\tau) |f(\tau, x(\tau))| \ud \tau
\\
& \le 
\kappa \dnorm{x}
+
G_{\max}
\int_a^b  \psi'(\tau) q(\tau) |\phi(x(\tau))| \ud \tau
\\
& \le 
\kappa \dnorm{x}
+
G_{\max} \dnorm{\psi'(\cdot) q(\cdot)}_{L^1(a, b)}
\phi\left(
   \int_a^b 
      \frac{|\psi'(\tau) q(\tau)|}{\dnorm{\psi'(\cdot) q(\cdot)}_{L^1(a, b)}} 
      |x(\tau)| 
   \ud \tau
\right)
\\
& \le 
\kappa \dnorm{x}
+
G_{\max} \dnorm{\psi'(\cdot) q(\cdot)}_{L^1(a, b)}
\phi(\dnorm{x})
,
\end{align*}
where
\begin{align*}
G_{\max}
=
\frac{1}{\Gamma(\alpha+\beta)}
(\psi(b)-\psi(a))^{\alpha+\beta-1}
\max\left\{
\beta^\beta (\alpha+\beta-1)^{\alpha+\beta-1}
,
\frac{\beta (\alpha+\beta-1)^{\frac{\alpha+\beta-1}{\beta}}}{(\alpha+2\beta-1)^{\frac{\alpha+2\beta-1}{\beta}}} 
\right\}
.
\end{align*}
From the latter inequality, we obtain 
\[
\int_a^b \psi'(\tau) q(\tau) \ud \tau
\ge 
\frac{1-\kappa}{G_{\max}}
\frac{\|x\|}{\phi(\|x\|)}
.
\]
This completes the proof of Theorem. 
\end{proof}

From the result of Theorem \ref{thm-Lya}, we immediately obtain 
Lyapunov-type inequalities for sequential fractional boundary value problems involving
Caputo and Hadamard fractional derivatives. 

\begin{corollary}
Let $0<\alpha, \beta \le 1$ with $\alpha+\beta>1$.
Suppose that Assumption $(\mathcal{A} 2)$ holds.
Suppose further that there exists $q: (a, b) \to \bR$ such that
\[
f(t, x)=q(t) x
\]
for all $t \in (a, b)$.
For $\psi(t)=t$, if $q(\cdot) \in L^1(a, b)$ 
and
the problem (\ref{main}) and (\ref{BV}) has a non-trivial mild solution,
then 
\[
\int_a^b 
   |q(\tau)|
\ud \tau
\ge 
\frac{1-\kappa}{G_1}
,
\]
where 
\begin{align*}
G_1
=
\frac{1}{\Gamma(\alpha+\beta)}
(b-a)^{\alpha+\beta-1}
\max\left\{
\beta^\beta (\alpha+\beta-1)^{\alpha+\beta-1}
,
\frac{\beta (\alpha+\beta-1)^{\frac{\alpha+\beta-1}{\beta}}}{(\alpha+2\beta-1)^{\frac{\alpha+2\beta-1}{\beta}}} 
\right\}
.
\end{align*} 
\end{corollary}

\begin{corollary}
Let $0<\alpha, \beta \le 1$ with $\alpha+\beta>1$.
Suppose that Assumption $(\mathcal{A} 2)$ holds.
Suppose further that there exists $q: (a, b) \to \bR$ $(a>0)$ such that
\[
f(t, x)=q(t) x
\]
for all $t \in (a, b)$.
For $\psi(t)=\ln t$, if 
\[
\int_a^b \frac{|q(\tau)|}{\tau} \ud \tau<+\infty
\] 
and
the problem (\ref{main}) and (\ref{BV}) has a non-trivial mild solution,
then 
\[
\int_a^b \frac{|q(\tau)|}{\tau} \ud \tau 
\ge 
\frac{1-\kappa}{G_2}
,
\]
where 
\begin{align*}
G_2
=
\frac{1}{\Gamma(\alpha+\beta)}
\left(\ln \frac{b}{a}\right)^{\alpha+\beta-1}
\max\left\{
\beta^\beta (\alpha+\beta-1)^{\alpha+\beta-1}
,
\frac{\beta (\alpha+\beta-1)^{\frac{\alpha+\beta-1}{\beta}}}{(\alpha+2\beta-1)^{\frac{\alpha+2\beta-1}{\beta}}} 
\right\}
.
\end{align*}
\end{corollary}

We can also use the obtained Lyapunov-type inequality to derive 
a lower bound for possible eigenvalues of the sequential fractional value problem. 
In fact, we have the following result. 

\begin{corollary}
Let $0<\alpha, \beta \le 1$ with $\alpha+\beta>1$, and $\psi \in C^1_+[a, b]$. 
Suppose that Assumption $(\mathcal{A} 2)$ holds.
Suppose further that $\lambda$ is an eigenvalue of the following problem
\[
\begin{cases}
\left(
   {^C}D_{a+}^{\alpha, \psi} 
   \,
   {^C}D_{a+}^{\beta, \psi}
   x
\right)
(t)
=
\lambda x(t)
, & a<t<b,
\cr
x(a)=0, \ \ x(b)=G(x)
.
\end{cases}
\]
Then
\[
|\lambda| 
\ge 
\frac{1-\kappa}{G_{\max} (\psi(b)-\psi(a))}
,
\]
where $G_{\max}$ defined in Proposition \ref{pro-max}.
\end{corollary}

Next, we present some existence results for our problem.
To this aim, we make the following assumptions.

\begin{itemize}
\item[•]
{\bf  Assumption $(\mathcal{A} 3)$:} 
There exist positive functions $\kappa_1, \kappa_2: (a, b) \to \bR_+$ and 
a positive non-decreasing function $\Psi: \bR_+ \to \bR_+$ such that
\begin{align*}
| f(t, x) |
& \le 
\kappa_1(t) \Psi(|x|)
\ \ 
\text{for any}
\ \
t \in (a, b), \ x \in \bR,
\\
| f(t, x) - f(t, y) |
& \le 
\kappa_2(t) |\Phi(x, y)|
\ \ 
\text{for any}
\ \
t \in (a, b), \ x, y \in \bR,
\end{align*}
where $\Phi \in C(\bR \times \bR, \bR)$ and $\Phi(x, y) \to 0$ as $|x-y| \to 0$.

\item[•]
{\bf  Assumption $(\mathcal{A} 4)$:} 
There exists a positive function $\kappa_3: (a, b) \to \bR_+$ such that
\begin{align*}
| f(t, x) - f(t, y) |
\le 
\kappa_3(t) |x-y|
\ \ 
\text{for any}
\ \
t \in (a, b), \ x, y \in \bR
\end{align*}

\item[•]
{\bf  Assumption $(\mathcal{A} 5)$:} 
There exists $\kappa_4 \in (0, 1)$ such that
\begin{align*}
|\gf(x)-\gf(y)
& \le 
\kappa_4 |x-y|
\ \ 
\text{for all}
\ \
x, y \in \bR
.
\end{align*}
\end{itemize}

In the following theorem, we introduce an existence result for our problem.

\begin{theorem}
\label{thm-exi}
Let $0<\alpha, \beta \le 1$ with $\alpha+\beta>1$, and let $\psi \in C^1_+[a, b]$.
Suppose that Assumptions $(\mathcal{A} 2)$ and $(\mathcal{A} 3)$ hold.
Suppose further that $\gf \in C(\bR, \bR)$ and $\psi'(\cdot) \kappa_i(\cdot) \in L^1(a, b)$ for $i=1, 2$.
If there exists $M>0$ such that
\[
M
> 
\frac{1}{1-\kappa}\Psi(M) G_{\max}
\dnorm{\psi'(\cdot) \kappa_1(\cdot)}_{L^1(a, b)}
\]
with $G_{\max}$ defined in Theorem \ref{thm-Lya}, then
the problem (\ref{main}) and (\ref{BV}) has at least one mild solution.
\end{theorem}

\begin{proof}
Let us define the operator $S: C[a, b] \to C[a, b]$ by
\begin{equation}
\label{S-df}
S x(t)
=
\left(\frac{\psi(t)-\psi(a)}{\psi(b)-\psi(a)}\right)^{\beta}
\gf(x)
+
\int_a^b G(\tau, t) \psi'(\tau) f(\tau, x(\tau)) \ud \tau
,
\end{equation}
where $G$ defined in Lemma \ref{lem-GR}.
We also denote
\begin{equation*}
B_r=\{x \in C[a, b]: \dnorm{x} \le r\}
\ \ \text{for some} \ \ r>0
.
\end{equation*}

We firstly show that $S$ is completely continuous operator.
We divide the process of proof into three steps.
\\
{\it Step 1.}
We verify that $S$ is continuous operator.
Assume that $x, y \in B_r$ for some $r>0$,
we obtain from Proposition \ref{pro-max} and Assumption $(\mathcal{A} 3)$ that
\begin{align*}
\dnorm{S x - S y}
& \le 
\sup_{a \le t \le b}
\left|
   \left(\frac{\psi(t)-\psi(a)}{\psi(b)-\psi(a)}\right)^{\beta}
   (\gf(x)-\gf(y))
\right| 
+
\sup_{a \le t \le b}
\left|
   \int_a^b G(\tau, t) \psi'(\tau) (f(\tau, x(\tau))-f(\tau, y(\tau))) \ud \tau
\right| 
\\
& \le 
|\gf(x)-\gf(y)|
+
G_{\max}
\int_a^b \psi'(\tau) \kappa_2(\tau) \snorm{\Phi(x(\tau), y(\tau))}\ud \tau
\\
& \le 
|\gf(x)-\gf(y)| 
+
G_{\max} \dnorm{\Phi(x(\cdot), y(\cdot))}
\dnorm{\psi'(\cdot) \kappa_2(\cdot)}_{L^1(a, b)}
\\
& \to 0
\ \ \text{as} \ \ x \to y
\end{align*}
due to $\gf \in C(\bR, \bR)$ 
and 
$\dnorm{\Phi(x(\cdot), y(\cdot))} \to 0$ as $x \to y$ in $C[a, b]$.
This completes the proof of {\it Step 1}.
\\
{\it Step 2.}
We prove that $S$ maps bound sets into bound sets in $C[a, b]$.
Suppose that $x \in B_r$ for some $r>0$.
By virtue of Proposition \ref{pro-max} and Assumption $(\mathcal{A} 3)$,
one has
\begin{align*}
\dnorm{S x}
& \le 
\sup_{a \le t \le b}
\left|
   \left(\frac{\psi(t)-\psi(a)}{\psi(b)-\psi(a)}\right)^{\beta}
   \gf(x)
+
   \int_a^b G(\tau, t) \psi'(\tau) f(\tau, x(\tau))\ud \tau
\right| 
\\
& \le 
|\gf(x)|
+
G_{\max}
\int_a^b \psi'(\tau) \kappa_2(\tau) \snorm{\Psi(x(\tau))}\ud \tau
\\
& \le 
\kappa r+\Psi(r) \dnorm{\psi'(\cdot) \kappa_1(\cdot)}_{L^1(a, b)}
.
\end{align*}
The proof of {\it Step 2} is completed.
\\
{\it Step 3.}
We verify that $S$ maps bound sets into equicontinuous sets of $C[a, b]$.
We can assume that $x \in B_r$ for some $r>0$.
For $a \le t_1 \le t_2 \le b$, using (\ref{inq}), Lemma \ref{pro-conti} and Assumption $(\mathcal{A} 3)$,
we have
\begin{align*}
\snorm{S x(t_2)-S x(t_1)}
& \le 
\left|
   \left(\frac{\psi(t_2)-\psi(a)}{\psi(b)-\psi(a)}\right)^{\beta}
   -
   \left(\frac{\psi(t_1)-\psi(a)}{\psi(b)-\psi(a)}\right)^{\beta}
\right|    
|\gf(x)|
+
\left|
   \int_a^b |G(\tau, t_2)-G(\tau, t_1)| \psi'(\tau) \kappa_1(\tau) |\Psi(x(\tau))| \ud \tau
\right| 
\\
& \le 
\kappa r \left(\frac{\psi(t_2)-\psi(t_1)}{\psi(b)-\psi(a)}\right)^{\beta}
+
\frac{2}{\Gamma(\alpha+\beta)}  
(\psi(t_2)-\psi(t_1))^{\alpha+\beta-1}
\Psi(r) \dnorm{\psi'(\cdot) \kappa_1(\cdot)}_{L^1(a, b)}
\\
& \to 0
\ \ \text{uniformly as} \ \ t_2 \to t_1
.
\end{align*} 
The proof of {\it Step 3} is completed.

We now prove the result of this theorem.
Using Assumption $(\mathcal{A} 3)$, we have
$|f(t, x(t))| \le \Psi(M) \kappa_2(t)$ for any $x \in B_M$ and $t \in (a, b)$.
By the same method to that used to prove {\it Step 2}, we get
\[
\dnorm{S x} 
\le 
\kappa M 
+
\Psi(M) G_{\max} \dnorm{\psi'(\cdot) \kappa_1(\cdot)}_{L^1(a, b)}
,
\]
where $G_{\max}$ defined in Theorem \ref{thm-Lya}.
If there exist $k \in (0, 1)$ and $x \in \partial B_M$ such that $x=k S x$, then
from the latter inequality, we get
\[
M 
\le 
\kappa M 
+
\Psi(M) G_{\max} \dnorm{\psi'(\cdot) \kappa_1(\cdot)}_{L^1(a, b)}
.
\]
Or,
\[
M \le \frac{1}{1-\kappa}
\Psi(M) G_{\max} \dnorm{\psi'(\cdot) \kappa_1(\cdot)}_{L^1(a, b)}
.
\]
We obtain a contradiction with the hypothesis of Theorem.
Hence, using Lemma \ref{Le-Sch}, we deduce that $S$ has a fixed point in $B_M$,
which is a mild solution of our problem.
The proof of Theorem is done.
\end{proof}

\begin{remark}
\label{rem-exi}
We emphasize that the obtained result in Theorem \ref{thm-exi}
works with some source functions including singularities.
For example, in Assumption $(\mathcal{A} 3)$, we consider the cases:
$
\kappa_i(t)
=
k_i (\psi(b)-\psi(t))^{-\gamma_i}
(\psi(t)-\psi(a))^{-\mu_i}
$
for some $k_i>0$ ($i=1, 2$).
If $\gamma_i, \mu_i<1$ 
then
we have
$\psi'(\cdot) \kappa_i(\cdot) \in L^1(a, b)$ for $i=1, 2$.
Moreover, if $k_1$ sufficiently close to zero, then we can find $M>0$ such that
\begin{equation*}
M
> 
\frac{1}{1-\kappa}\Psi(M) G_{\max}
\dnorm{\psi'(\cdot) \kappa_1(\cdot)}_{L^1(a, b)}
.
\end{equation*}
To verify the above statements, we may use the fact that
\begin{align*}
\int_a^b 
   \psi'(\tau) (\psi(b)-\psi(\tau))^{-\gamma}
   (\psi(\tau)-\psi(a))^{-\mu}
\ud \tau
=
(\psi(b)-\psi(a))^{1-\gamma-\mu}
B(1-\gamma, 1-\mu)
,
\end{align*}
where $\gamma, \mu<1$ and $B(\cdot, \cdot)$ is the Beta function.
\end{remark}

Finally, we give an existence and uniqueness result for problem.

\begin{theorem}
\label{thm-uq}
Let $0<\alpha, \beta \le 1$ with $\alpha+\beta>1$, and let $\psi \in C^1_+[a, b]$.
Suppose that Assumptions $(\mathcal{A} 4)$ and $(\mathcal{A} 5)$ hold.
Suppose further that $\psi'(\cdot) \kappa_3(\cdot) \in L^1(a, b)$ and $\psi'(\cdot) f(\cdot, 0) \in L^1(a, b)$.
If 
\[
 \kappa_4+ G_{\max}\dnorm{\psi'(\cdot) \kappa_3(\cdot)}_{L^1(a, b)}
 <1
\]
then the problem (\ref{main}) and (\ref{BV}) has a unique mild solution.
\end{theorem}

\begin{proof}
Let us consider the operator $S$ defined in (\ref{S-df}).
We can find from Assumption $(\mathcal{A} 4)$ that $|f(t, x)| \le \kappa_3(t)|x|+|f(t, 0)|$.
Therefore, we can use the same method to that used to prove {\it Step 3} 
in the proof of Theorem \ref{thm-exi} to verify that $S$ is well-defined.
\\
Next, we show that $S$ is a contraction.
Indeed, for $x, y \in C[a, b]$, 
using Assumptions $(\mathcal{A} 4)$ and $(\mathcal{A} 5)$,
similar to {\it Step 1} in the proof of Theorem \ref{thm-exi},
we obtain 
\begin{align*}
\dnorm{Sx-Sy}
& \le 
|\gf(x)-\gf(y)| 
+
G_{\max} \dnorm{\psi'(\cdot) \kappa_3(\cdot)}_{L^1(a, b)}
\dnorm{x-y}
\\
& \le 
\left(\kappa_4+G_{\max} \dnorm{\psi'(\cdot) \kappa_3(\cdot)}_{L^1(a, b)}\right)
\dnorm{x-y}
.
\end{align*}
This shows that $S$ is a contraction 
due to $\kappa_4+ G_{\max}\dnorm{\psi'(\cdot) \kappa_3(\cdot)}_{L^1(a, b)} <1$.
Hence, we conclude that $S$ has a unique fixed point in $C[a, b]$,
which is the mild solution of problem (\ref{main}) and (\ref{BV}).
This completes the proof of Theorem. 
\end{proof}

\begin{remark}
If 
$
k_3(t)
=
k_3 (\psi(b)-\psi(t))^{-\gamma_3}
(\psi(t)-\psi(a))^{-\mu_3}
$
and 
$
|f(t, 0)|
\le 
k_4 (\psi(b)-\psi(t))^{-\gamma_4}
(\psi(t)-\psi(a))^{-\mu_4}
$
for some $k_i>0$ and $\gamma_i, \mu_i<1$ for $i=3,4$,
similar to Remark \ref{rem-exi}, we can verify that 
$\psi'(\cdot) \kappa_3(\cdot) \in L^1(a, b)$ and $\psi'(\cdot) f(\cdot, 0) \in L^1(a, b)$.
Moreover, if $k_3$ sufficiently close to zero, then 
$\kappa_4+ G_{\max}\dnorm{\psi'(\cdot) \kappa_3(\cdot)}_{L^1(a, b)} <1$.
Hence, we can conclude that our result in Theorem \ref{thm-uq} holds for some source functions
including singularities.
\end{remark}

\section*{Acknowledgements}
We would like to thank the handling editor and the anonymous reviewers for their valuable comments and suggestions which helped us to improve the paper.

\end{document}